\documentclass[reqno,12pt]{article}
\usepackage{amsmath,amsthm,amssymb,bm} 
\usepackage{datetime}

\newtheorem{theorem}{Theorem}
\newtheorem{lemma}[theorem]{Lemma}
\newtheorem{proposition}[theorem]{Proposition}
\newtheorem{corollary}[theorem]{Corollary}

\def\E{{\mathbb E}}

\def\P{{\mathbb P}}
\def\R{{\mathbb R}}

\def\disc{{\mathop{\rm disc}}}
\def\tr{{\mathop{\rm tr}}}

\begin{document}
\title{Intersections of hypergraphs}
\author{
B\'ela Bollob\'as
\footnote{Department of Pure Mathematics and Mathematical Statistics, 
Wilberforce Road, Cambridge CB3 0WB, UK {\em and}  Department of
Mathematical Sciences, University of Memphis, Memphis TN38152,
USA; email: bollobas@msci.memphis.edu.  Research supported in part
by 
University of Memphis Foundation grant
UMF 20953,
NSF grant DMS-0906634, 
and
DARPA grant 9060-200241 CLIN 01.} 
\and
Alex Scott
\footnote{Mathematical Institute, 24-29 St Giles', Oxford, OX1 3LB, UK;
email: scott@maths.ox.ac.uk.}
}
\date{}
\maketitle

\begin{abstract}
Given two weighted $k$-uniform hypergraphs $G$, $H$ of order $n$, how 
much (or little) can we make them overlap by placing them on the same vertex set?  
If we place them at random, how concentrated is the distribution of the intersection?
The aim of this paper is to investigate these questions.
\end{abstract}

\section{Introduction}\label{intro}

The discrepancy of a set of points in a subset of Euclidean space measures how uniformly the points are spread through the set.  
For instance,  the discrepancy of a set of $n$ points in a square of area $n$ can be defined as the maximum difference
between the area of a subsquare and the number of points from the set that it contains.
Discrepancy theory in the geometric setting has been studied for almost a century, since the work of Weyl \cite{W16} on sequences, and is of interest in 
areas including number theory and combinatorics, as well as having applications in computational geometry and numerical integration
(see for instance the books by Beck and Chen \cite{BC87},
Kuipers and Niederreiter \cite{KN74} 
and Drmota and Tichy \cite{DT97}).

In the discrete context, a similar notion of discrepancy for hypergraphs was introduced forty years ago by Erd\H os and Spencer \cite{ES71}, 
and measures the extent to which the edges of a hypergraph are uniformly distributed (inside the complete graph).
Erd\H os and Spencer showed that the edges of a $k$-uniform hypergraph can not be distributed too uniformly:
for every $k$-uniform hypergraph on $n$ vertices, there is a subset $S$ in which the number of edges 
differs from $\frac 12\binom {|S|}{k}$  by at least
$c_kn^{(k+1)/2}$. (This bound is optimal up to a constant factor.)
In the case of graphs (i.e.~$k=2$), Erd\H os, Goldberg, Pach and Spencer \cite{EGPS88} 
later extended this to graphs of any density $p$, where the measure of discrepancy is the maximum difference between the number of edges in a subset $S$ and the expected $p\binom{|S|}{2}$.
(There are a number of other standard ways to measure discrepancy for discrete structures:
see Beck and S\'os \cite{BS95},
Chazelle \cite{C00}
and Matou\v sek \cite{M99}.) 

The aim of this paper is to study the discrepancy of {\em pairs} of hypergraphs.
The discrepancy of a pair hypergraphs, introduced in \cite{BS11},
measures the extent to which the edges of the two hypergraphs are uniformly and independently distributed.  
Given $k$-uniform hypergraphs $G$ and $H$ with $n$ vertices and densities $p$, $Q$, the discrepancy of the pair $G$, $H$ is the maximum size, 
over all bijections between their vertex sets,
of the difference between their intersection and $pq\binom nk$ (the expected intersection under a random mapping).  
For instance, $G$ and $H$ have discrepancy 0 if their intersection has the same size for any placement of both hypergraphs onto the same vertex set; on the other hand, if $G$ and $H$ are isomorphic to the same incomplete graph then their discrepancy will be large, as any isomorphism between them will give a much larger than average intersection.

In light of the results of Erd\H os and Spencer \cite{ES71}, it is natural to expect that every pair of (unweighted) $k$-uniform hypergraphs of moderate density should have large discrepancy (of order $n^{(k+1)/2}$), and we conjectured in \cite{BS11} that this should be the case.  
For $k=2$, this conjecture was proved in \cite{BS11}, but for $k=3$ it turns out that there is a counterexample (see section \ref{two}); 
for $k\ge4$, the conjecture is still open.
In this paper, we investigate the discrepancy of pairs of {\em weighted} hypergraphs. 
It turns out that, for weighted hypergraphs the picture is dramatically different from the unweighted case:
\begin{itemize}
\item  For every $k\ge1$, there is a set of $k$ nontrivial weighted $k$-uniform hypergraphs such that every pair has discrepancy 0. 
\end{itemize}
On the other hand, if we take one additional hypergraph, there must be a pair with large discrepancy:
\begin{itemize}
\item  For every $k\ge1$, and every set of $k+1$ nontrivial normalised weighted hypergraphs, there is some pair that
has discrepancy at least $c_kn^{(k+1)/2}$.
\end{itemize}

As we shall see in Section \ref{two}, both results are special cases of much more general results (Theorem \ref{orthogonal} and Theorem \ref{thdisc}, respectively) on
the discrepancy of pairs of hypergraphs.  
We will also be interested in the size of the intersection when two (weighted) $k$-uniform hypergraphs are placed at random onto the same vertex set.  
For sequences ($k=1$) and graphs ($k=2$), the distribution of this intersection has been extensively studied in the statistical literature, and central limit theorems have been proved under various conditions.
Here, we work with general $k$, but prove only a lower bound (Theorem \ref{thexp}) on the concentration of the distribution.
 
The rest of the paper is organized as follows: after 
giving some background in Section \ref{one}, 
we discuss the discrepancy of pairs of weighted hypergraphs and
present our results in Section \ref{two}. 
We give notation and some useful tools in Section \ref{tools}, 
and define the $W$-vector in Section \ref{wvector}.
We study the effects of a single transposition in Section \ref{trans}; we prove
Theorems \ref{thdisc} and \ref{thexp} in Section \ref{proofs1}; 
and Theorem \ref{orthogonal} is proved in Section \ref{proofs2}.  
We conclude in Section \ref{conc} with some comments and open problems.

We work throughout the paper with weighted hypergraphs.  
A {\em weighted $k$-uniform hypergraph with vertex set $V$} is a function $w:V^{(k)}\to\R$, i.e.~a weighting on the
$k$-sets in $V$.
An {\em unweighted $k$-uniform hypergraph} is a subset of $V^{(k)}$, and can be identified with the weighted
hypergraph given by the indicator function for its edges.  
The {\em density} of $w$ is $d(w)=w(V)/\binom nk$,
where $w(V)=\sum_{e\in V^{(k)}}w(e)$.

\subsection{Discrepancy of a single hypergraph}\label{one}
In this section we give some background on the 
discrepancy of a single hypergraph. 

If $S\subset V$ is chosen uniformly at random from all sets of some fixed size, we have 
\begin{equation}\label{expect0}
\E w(S)=d(G)\binom{|S|}{k}.
\end{equation}
It therefore makes sense to define 
the {\em discrepancy} of $G$ by 
\begin{equation}\label{discdef}
\disc(G)=\max_{S\subset V}\left|w(S)-d(G)\binom {|S|}k\right|.
\end{equation}
The discrepancy measure how far $w(S)$ can deviate from \eqref{expect0}, but does not indicate whether the number of edges
is greater or less than we expect.  We therefore define
the {\em positive discrepancy} $\disc^+(G)$ by
$$\disc^+(G)=\max_{S\subset V} \left(w(S)-d(G)\binom{|S|}{k}\right).$$
and the {\em negative discrepancy} $\disc^-(G)$ by
$$\disc^-(G)=\max_{S\subset V} \left(d(G)\binom{|S|}{k}-w(S)\right).$$
Clearly $\disc(G)=\max\{\disc^+(G),\disc^-(G)\}$, and
it follows from \eqref{expect0} that both positive and negative discrepancy are nonnegative.  
(We note that notions of signed discrepancy have been considered in other contexts: see 
Erd\H os, Faudree, Rousseau and Schelp \cite{EFRS94}, 
Krivelevich \cite{K95}
and Keevash and Sudakov \cite{KS03}.)

The discrepancy of graphs and hypergraphs was introduced by Erd\H os and Spencer \cite{ES71}, who showed 
that every $k$-uniform hypergraph $G$ of order $n$ and density $1/2$ has 
\begin{equation}\label{esresult}
\disc(G)\ge c_kn^{(k+1)/2}.
\end{equation}
For $k=2$ (i.e.~ for graphs), Erd\H os, Goldberg, Pach and Spencer \cite{EGPS88} extended \eqref{esresult} to 
arbitrary density, showing that if $G$ has order $n$ and density $p$, where $p\in (2/(n-1),1-2/(n-1))$, then
\begin{equation}\label{egps}
\disc(G)\ge c\sqrt{p(1-p)}n^{3/2}.
\end{equation}

By considering random graphs in $\mathcal G(n, 1/2)$.
it can be seen that the discrepancy of a graph on $n$ vertices
can be as small as $O(n^{3/2})$;
thus \eqref{esresult} is optimal up to the constant. However, the one-sided
discrepancies can be smaller: 
$K_{n/2,n/2}$ has positive discrepancy $O(n)$,
while its complement $2K_{n/2}$ has negative discrepancy $O(n)$;
on the other hand, both graphs have discrepancy $\Omega(n^2)$ in the other direction. 
Bollob\'as and Scott \cite{BS06} showed that this tradeoff is unavoidable:
for every graph $G$ of order $n$, with $p\binom n2$ edges, where $p(1-p)\ge1/n$, we have
\begin{equation}\label{prod}
 \disc^+(G)\disc^-(G)\ge cp(1-p)n^3.
\end{equation}
Note that \eqref{egps} follows immediately.
A similar result to \eqref{prod} holds for $k$-uniform hypergraphs \cite{BS06}: for every hypergraph $H$ of order $n$ and
density $p$, where $p(1-p)\ge1/n$, 
\begin{equation}\label{bsresult}
\disc^+(H)\disc^-(H)\ge c_kp(1-p)n^{k+1}.
\end{equation}

\subsection{Results}\label{two}

We now turn to the discrepancy of pairs of hypergraphs.
Given two weighted hypergraphs $w, u$ on $V$, 
the {\em intersection} of $w$ and $u$ is naturally defined as $\langle w,u\rangle$, where 
$\langle\cdot,\cdot\rangle$ is the standard inner product on $V^{(k)}$. 
There is also a natural action of 
the symmetric group $S(V)$ on the space of weighted hypergraphs, given by $w_\pi(e)=w(\pi^{-1}e)$
(see Section \ref{tools} for notation).

If we permute $w$ uniformly at random, the expected intersection with $u$ is 
\begin{equation}\label{d0}
\E_\pi\langle w_\pi,u\rangle=d(w)d(u)\binom nk.
\end{equation}
This leads us to define the {\em positive discrepancy of the pair $w$, $u$} by
\begin{equation}\label{wdplus}
\disc^+(w,u)=\max_\pi \langle w_\pi,u\rangle-d(w)d(u)\binom nk
\end{equation}
and the {\em negative discrepancy} by
\begin{equation}\label{wdminus}
\disc^-(w,u)
=d(w)d(u)\binom nk-\min_\pi \langle w_\pi,u\rangle.
\end{equation}
Note that both are nonnegative, by \eqref{d0}.
The {\em discrepancy} $\disc(w,u)$ is then defined as
$$\disc(w,u)=\max\{\disc^+(w,u),\disc^-(w,u)\}
=\max_\pi|\langle w_\pi,u\rangle-d(w)d(u)\binom nk|.$$  

The discrepancy of a pair of hypergraphs
was introduced in \cite{BS11}, and is a natural extension of the notion of discrepancy for a single hypergraph.
Analogously with \eqref{egps}, it was shown in \cite{BS11} that, for every
pair of graphs $G$, $H$, of order $n$ and densities  
$p,q \in (16/n,1- 16/n)$, 
\begin{equation}\label{noprod}
\disc(G,H)\ge c(p,q)n^{3/2},
\end{equation}
where $c(p,q)=p^2 (1 - p)^2 q^2 (1 - q)^2/10^{10}$.

As with the discrepancy of a single graph, the one-sided discrepancies of pairs of graphs can be quite small.  
For instance, consider $G=K_{n/2,n/2}$ and $H=2K_{n/2}$:
this pair has positive discrepancy
$O(n)$, which is minimal up to a constant factor for dense graphs  
(although the negative discrepancy is $\Omega(n^2)$,
which is maximal up to a constant factor).
However, it was shown in \cite{BS11} that
there is a bound on the {\em product} of the two discrepancies:
for every
pair of graphs $G$, $H$, of order $n$ and densities  
$p,q \in (16/n,1- 16/n)$, 
\begin{equation}\label{prod2}
\disc^+(G, H)\disc^-(G, H) \ge c(p,q)^2 n^3.
\end{equation}
Thus if the discrepancy on one side is small, the discrepancy on the other must be large.
The bound \eqref{prod2} is sharp up to the constant, as can be seen
from $K_{n/2,/n/2}$ and $2K_{n/2}$ or by taking $G=2K_{n/2}$ and letting
$H$ be a random graph with fixed density.  
Note also that \eqref{prod} is a special case of \eqref{prod2}, as we can take $H=K_{n/2}\cup(n/2)K_1$
(which corresponds to restricting $S$ to have size $n/2$ in \eqref{discdef}).  
Equation \eqref{noprod} also follows as an immediate corollary.

It seems natural to expect that bounds similar to \eqref{noprod} and \eqref{prod2} 
should hold for $k$-uniform hypergraphs: by analogy with the situation for a single hypergraph
(see \eqref{esresult} and \eqref{bsresult} above), 
we should expect a lower bound of form $cn^{k+1}$ on the product of positive and negative discrepancies, 
which would in turn yield a bound of form $cn^{(k+1)/2}$ on the (unsigned) discrepancy.  
Such a bound was conjectured in \cite{BS11}, but we were surprised to find the following 
simple counterexample for 3-uniform hypergraphs.  Let $V$ be a set of $n$ vertices, and let $V=A\cup B$ be a partition.  We let $G$ be the 3-uniform hypergraph on $V$ with all triples that meet both $A$ and $B$, and $H$ be a Steiner triple system.  Then $\disc(G,H)=0$.  
(This is easily shown: in any placement of $H$, there must be exactly $|A| |B|/2$ edges of $H$ that meet both $A$ and $B$, as each such edge contains exactly two edges from $\{ab:a\in A,b\in B\}$.)  But now we can obtain an example in which both hypergraphs have density bounded away from 0 and 1 by taking $H$ to be the union of a suitable number of edge-disjoint disjoint Steiner triple systems (see
Doyen \cite{D72} or Teirlinck \cite{T73} for constructions).

For {\em weighted} hypergraphs, the situation is even more dramatic: 
there is a nontrivial set of $k$ weighted $k$-uniform hypergraphs for which every pair has discrepancy 0.
Note that if $w$ is a constant function, then trivially $\disc(w,u)=0$ for every $u$.  Indeed, if we add a constant function to 
$w$ it does not affect the discrepancy (that is, $\disc(w+\lambda\mathbf1,u)=\disc(w,u)$).  So, to avoid triviality,
we will restrict ourselves 
to hypergraphs $w$ such that $w(V)=0$.  We then have the following result.

\begin{theorem}\label{orthogset}
Let $k\ge 2$.  For every $n\ge 2k$ there are weighted hypergraphs $w_1,\ldots,w_k$ with vertex set $[n]$ such that 
$w([n])=0$  and $||w_i||_1=\binom nk$ for every $i$ and, for $0\le i<j\le k$ we have
$$\disc (w_i,w_j)=0.$$ 
\end{theorem}

Theorem \ref{orthogset} is a special case of a much stronger result below (Theorem \ref{orthogonal}), which gives a description of all pairs of weighted hypergraphs
with discrepancy 0, and allows us to characterize collections of weighted hypergraphs satisfying Theorem \ref{orthogset}.

If we have $k+1$ weighted hypergraphs, however, the picture is very different: we do get a version of \eqref{prod2} for at least one pair, and attain the bound conjectured in \cite{BS11}.

\begin{theorem}\label{thma0}
For every $k\ge1$ there are constants $c,c'>0$ such that the following holds.
Let $n\ge 2k$, and suppose that $w_1,\ldots,w_{k+1}$
are weighted $k$-uniform hypergraphs on $[n]$ such that $w_i([n])=0$  and $||w_i||_1=\binom nk$ for every $i$.
Then there are distinct $i$ and $j$ such that
$$\disc^+(w_i,w_j)\disc^-(w_i,w_j) \ge cn^{k+1}.$$
In particular, there are $i<j$ such that,
$$\disc(w_i,w_j) \ge c'n^{(k+1)/2}.$$
\end{theorem}

We will also prove (Theorem \ref{thma1}) that every family $\mathcal F$ of weighted $k$-uniform hypergraphs $w$ with $w(V)=0$ can
be partitioned into $k$ families of hypergraphs with pairwise large discrepancy.

Theorems \ref{thma0} and \ref{thma1} both follow from a much stronger quantitative result (Theorem \ref{thdisc}), 
which will allow us to prove a lower bound on the discrepancy 
of a pair of weighted hypergraphs.  
In order to state this result, we need to introduce the {\em $W$-vector} of a weighted hypergraph
(the formal definition will require a little work, so we defer it to Section \ref{wvector}).
For every weighted $k$-uniform hypergraph $w$, we will define a sequence of $k+1$ nonnegative weights 
$W_0,\ldots,W_k$, giving us the $W$-vector $W=(W_0,\ldots,W_k)$.
As we shall see in Lemma \ref{thmb},
it turns out that the $W$-vector preserves the weight of $w$, in that there are constants $c,c'$ such that
\begin{equation}\label{nkahead}
c||w||_1/n^k \le \sum_{i=0}^kW_i\le c'||w||_1/n^k.
\end{equation}  
In particular, if $||w_i||_1=\binom nk$ then some component of the $W$-vector is at least a constant.

We can now state a quantitative version of Theorem \ref{thma0}. 

\begin{theorem}\label{thdisc}
For every $k\ge1$ there are $c,c'>0$ such that the following holds.
For every $n\ge k$ and every pair of weighted hypergraphs $w,u:[n]^{(k)}\to\R$, we have
$$\disc^+(w,u)\disc^-(w,u)\ge cn^{2k+1}\sum_{i=1}^k n^{-i}W_i^2U_i^2,$$
where $(W_0,\ldots,W_k)$ and $(U_0,\ldots,U_k)$ are the $W$-vectors of $w$ and $u$ respectively.
In particular, 
$$\disc(w,u)\ge c'n^{k+1/2}\sum_{i=1}^kn^{-i/2}W_iU_i.$$
\end{theorem}

Theorem \ref{thdisc} bounds the discrepancy of a pair $w$, $u$ of weighted $k$-uniform hypergraphs in terms of the dot product of their $W$-vectors.   Note that, as $W$-vectors belong to $\R^{k+1}$, it does not exclude the possibility that we could have $k+1$ nontrivial hypergraphs that pairwise have discrepancy 0 (so Theorem \ref{orthogset} is consistent with Theorem \ref{thdisc}).
However, in light of \eqref{nkahead}, for any collection of $k+1$ hypergraphs $w_1,\ldots,w_{k+1}$ as in Theorem \ref{thma0}, each $w_i$ must have at least constant weight in some component of its $W$-vector.  Since all the $w_i$
have total weight $0$, it will follow from the definition of $W$-vectors (in particular, from \eqref{wv1}) that the 
$W$-vectors of the $w_i$ all have first component 0.  Since there are $k$ remaining components,
some pair $w_i$, $w_j$ must have constant weight in the same component.  Theorem \ref{thma0} then follows immediately from  Theorem \ref{thdisc}.

In addition to bounding the discrepancy, we will also prove a result on the
expectation of the intersection $|\langle w_\pi,u\rangle|$ of two weighted $k$-uniform hypergraphs, when $\pi$ is chosen uniformly at random.

For $k=1$ (i.e.~sequences),
the distribution of $\langle w_\pi,u\rangle$ has been extensively studied.
Wald and Wolfowitz \cite{WW44}
proved a central limit theorem for $\langle w_\pi,u\rangle$ (under suitable conditions),
and subsequent generalizations were given by 
Noether \cite{N49},
Hoeffding \cite{H51}, 
Dwass \cite{D53} and many other authors.
For $k=2$ (i.e.~graphs), 
random intersections
$\langle w_\pi,u\rangle$
arise naturally in a number of statistical applications
(for instance, Barbour and Chen \cite{BC05} mention
applications in geography and epidemiology:
see 
Moran \cite{M48}, Geary \cite{G54}, Knox \cite{K64}, Mantel \cite{M67} and Hubert \cite{H87}).
The distribution of $\langle w_\pi,u\rangle$
has been considered by many authors starting with Daniels \cite{D44}, and 
including 
Barton and David \cite{BD66},
Abe \cite{A69}, Barbour and Eagleson \cite{BE86,BE86b}
and  Barbour and Chen \cite{BC05}, and there are sophisticated central limit theorems.

In this paper, we consider general $k$, but do not determine the limiting distribution of $|\langle w_\pi,u\rangle|$.  
However, we give a weak bound on the concentration of the distribution of 
$|\langle w_\pi,u\rangle|$, by bounding
the expected value of $|\langle w_\pi,u\rangle|$.

\begin{theorem}\label{thexp}
For every $k\ge1$ there is $c>0$ such that, for every $n>2k$ and every
pair of weighted hypergraphs $w,u:[n]^{(k)}\to\R$,
\begin{equation}\label{expineq}
\E_\pi|\langle w_\pi,u\rangle|\ge cn^k\sum_{i=0}^k n^{-i/2}W_iU_i,
\end{equation}
where $(W_0,\ldots,W_k)$ and $(U_0,\ldots,U_k)$ are the $W$-vectors of $w$ and $u$ respectively.
\end{theorem}

\section{Notation and tools}\label{tools}

We use standard notation: $V^{(k)}$ denotes the collection of $k$-sets in $V$;  we shall often refer to these as {\em edges}.  
We write $[n]=\{1,\ldots,n\}$. 
For any function $f$, we write $f^+(x)=\max\{f(x),0\}$ and $f^-(x)=\max\{-f(x),0\}$.  

A {\em weighted $k$-uniform hypergraph with vertex set $V$} is simply a function $w:V^{(k)}\to\R$.   
For $S\subset V$, we define $w(S)=\sum_{e\in S^{(k)}}w(e)$.  Given weighted $k$-uniform hypergraphs $w$, $u$ on vertex set $V$, we
define a standard norm and inner product: $||w||_1=\sum_{e\in V^{(k)}}|w(e)|$ 
and $\langle w,u\rangle=\sum_{e\in V^{(k)}}w(e)u(e)$.  
The {\em density} of $w$ is $d(w)=w(V)/\binom{|V|}{k}$.
We also define the constant
function $\mathbf 1$ by $\mathbf 1(e)=1$ for every edge $e$.
We will feel free to move without comment between a hypergraph $H$, and the corresponding weight function $w=w_H$ defined by
$w(e)={\mathbf 1}(e\in E(H))$.

There is a natural action of permutations of $V$ on weighted hypergraphs.  Given a function $f:V^{(k)}\to \R$ and a permutation $\pi$ of $V$, we define the function $f_\pi$
by $f_\pi(e)=f(\pi^{-1}(e))$.  Thus for permutations $\pi$, $\rho$, we have $f_{\pi\rho}=(f_\rho)_\pi$,
as $f_{\pi\rho}(e)=f((\pi\rho)^{-1}e)=f(\rho^{-1}\pi^{-1}e)=f_\rho(\pi^{-1}e)=(f_\rho)_\pi(e)$.

We say that weighted $k$-uniform hypergraphs $w$ on vertex set $V$ and $u$ on vertex set $U$ are {\em isomorphic} if there is
a bijection $f:V\to U$ such that $u(f(e))=w(e)$ for every edge $e\in V^{(k)}$.  Clearly $w$ and $w_\pi$ are isomorphic for any $\pi\in S(V)$. 

For weighted hypergraphs $w$, $u$, the {\em positive discrepancy} $\disc^+(w,u)$ and {\em negative discrepancy} $\disc^-(w,u)$ 
are defined as in \eqref{wdplus} and \eqref{wdminus}; we then set $\disc(w,u)=\max(\disc^+(w,u),\disc^-(w,u))$. 

Throughout the paper we will take expectations over randomly chosen vertices or edges.  Unless otherwise specified, this will always be with respect to the uniform distribution.
We will also adopt the convention that $\E'$ and ${\sum}'$ denote expectation and sum over {\em distinct} choices of argument: for instance if we are choosing random vertices from $V$, then $\E_{x,y}$ denotes the
expectation over the $|V|^2$ possible choices of an ordered pair $(x,y)$, while $\E_{x,y}'$ denotes the expectation over the 
$|V|(|V|-1)$ possible ordered pairs $(x,y)$ such that $x\ne y$, with respect to the uniform distribution in both cases.
Finally, if we take expectations with respect to a permutation $\pi$, then unless stated otherwise this will always be taken to be chosen uniformly at random from the symmetric
group $S(V)$ on $V$.

It will be useful to note a few elementary facts.

\begin{lemma}\label{poly}
Let $k\ge1$ be fixed.  There is a constant $c_k>0$ such that every polynomial $f(x)=\sum_{i=0}^ka_ix^i$ with $\max_i|a_i|=1$ satisfies
\begin{equation}\label{poly2}
\int_0^1|f(x)|dx\ge c_k.
\end{equation}
In particular, this implies
\begin{equation}\label{poly1}
\sup_{x\in[0,1]}|f(x)|\ge c_k.
\end{equation}
\end{lemma}

\begin{proof}
The proof is straightforward.  For ${\mathbf a}=(a_0,\ldots,a_k)\in[-1,1]^{k+1}\setminus(-1,1)^{k+1}$, let
$F(\mathbf a)=\int_0^1|\sum_{i=0}^ka_ix^i|dx$.  Then $F$ is continuous and strictly positive, and so we are done by compactness.
\end{proof}

The following simple bound is proved in \cite{BS06}.

\begin{lemma}\label{sumn}
Let ${\mathbf a}=(a_i)_{i=1}^{n}$ be a sequence of real numbers and $I\subset \{1,\ldots,n\}$ a subset chosen uniformly at random.  Then
$$\E|\sum_{i\in I}a_i|\ge\frac{||\mathbf a||_1}{\sqrt{8n}}.$$
\end{lemma}

It will also be useful to note the following elementary fact.

\begin{proposition}\label{simpleprop}
If $X$ is a random variable with $\E X=0$, and $c\in \R$, then $\E|X+c|\ge\max\{\E|X|/2,|c|\}$.
\end{proposition}

\begin{proof}
We may assume $c>0$.  We have $\E X^+=\E X^-=\E |X|/2$.  But
$\E|X+c|\ge\E[(X+c){\mathbf 1}_{(X\ge0)}]\ge\E[X{\mathbf 1}_{(X\ge0)}]=\E X^+=\E |X|/2$.  Also,
$\E|X+c|\ge\E(X+c)=\E X+c=c$.
\end{proof}

\section{The $W$-vector}\label{wvector}

Given a weighted hypergraph $w:V^{(k)}\to\R$, where $V$ is a set of size $n\ge 2k\ge 0$,
we define in this section a corresponding {\em $W$-vector} $(W_0,\ldots,W_k)$, where each $W_i$ is a nonnegative real.  

We 
start by defining
\begin{equation}\label{wv1}
W_0=|d(w)|=|w(V)|/\binom nk=\binom{n}{k}^{-1}|\sum_{e\in V^{(k)}}w(e)|=|\E_ew(e)|,
\end{equation}
where we write $\E_e$ for the expectation over an edge $e$ chosen uniformly at random over all $\binom nk$ possibilities.
Clearly $W_0=0$ if and only if $w(V)=0$.

For $i\ge1$, we define $W_i$ recursively.  For each $\{x,y\}\in V^{(2)}$, 
the {\em difference weighting} $w^{xy}$ is defined on 
sets $e\in(V\setminus\{x,y\})^{(k-1)}$ by
$$w^{xy}(e)=w(e\cup \{x\})-w(e\cup\{y\}).$$
Note that $w^{xy}=-w^{yx}$.
For any choice of distinct $x$ and $y$, the difference weighting $w^{xy}$ has a $W$-vector $(W_0^{xy},\ldots,W_{k-1}^{xy})$.  
We can therefore define, for $1\le i\le k$,
\begin{equation}\label{wv2}
W_i=\frac{1}{n(n-1)}{\sum_{x,y}}'W_{i-1}^{xy}=\E_{x,y}'W_{i-1}^{xy},
\end{equation}
where, as usual, we write $\E'$ and $\sum'$ for the expectation and sum over distinct indices.
Note that the $W$-vector is well-defined, 
as the $W$-vector for a weighting of $k$-sets is given in terms of the $W$-vectors for weightings of 
various collections of $(k-1)$-sets.  

For example, in the trivial case $k=0$, a weighting is just a constant $w$, 
and the $W$-vector is $(W_0)$, where $W_0=|w|$.  
For $k=1$, we have a weight function $w:V\to\R$.  
If $|V|=n$, we have $W_0=|\sum_{v\in V}w(v)|/n$.  Now for distinct $x,y\in V$, $w^{xy}$ is
a weighting on the $(k-1)$-sets, which in this case is just a weight (on the empty set) given by
$$w^{xy}(\emptyset)=w(\{x\}\cup\emptyset)-w(\{y\}\cup\emptyset)=w(x)-w(y),$$
so $w^{xy}$ has $W$-vector given by $W^{xy}_0=W^{yx}_0=|w(x)-w(y)|$.
We then have
\begin{equation}\label{w1def}
W_1=\frac{1}{n(n-1)}{\sum_{x,y}}'W^{xy}_0=\E'_{x,y}|w(x)-w(y)|.
\end{equation}

Defining the $W$-vector by \eqref{wv1} and \eqref{wv2} will be helpful in some of
the proofs below.  However, we now give a second approach 
that allows us to write the $W$-vector in a form that is frequently more convenient.

We begin by choosing an arbitrary sequence $x_1,y_1,\ldots,x_k,y_k$ of $2k$ distinct vertices in $V$.  
For $i=1,\ldots,k$, we define $Y_i=\{y_1,\ldots,y_i\}$ and
${\mathbf s_i}=(x_1,y_1,\ldots,x_i,y_i)$.  We say that a set $A\in V^{(k)}$ is {\em compatible with ${\mathbf s_i}$} 
if $|A\cap\{x_j,y_j\}|=1$ for $j=1,\ldots,i$.  We define
weighted $k$-uniform hypergraphs $\phi_i$ and $\phi_i^*$ by
\begin{equation}\label{canonical}
\phi_i(A)=\phi_i(x_1,y_1,\ldots,x_i,y_i;A)=
\begin{cases}
(-1)^{|A\cap Y_i|}	&	\mbox{$A$ compatible with $s_i$}\\
0										&	\mbox{otherwise.}
\end{cases}
\end{equation}
and 
\begin{equation}\label{canonicalstar}
\phi_i^*=\phi_i\big/\binom{n-2i}{k-i}.
\end{equation}
Note that we have normalized so that $||\phi_i^*||_1=\Theta(1)$.

The definitions of $\phi_i$ and $\phi_i^*$ depend on the sequence of vertices we pick for 
$x_1,y_1,\ldots,x_i,y_i$.  However, different choices give isomorphic weightings, and in practice we will always symmetrize over permutations of the vertices, as in \eqref{nice2} below, so our results do not depend on our particular choices.

\begin{lemma}\label{niceform}
Let $n\ge 2k\ge 1$, and suppose that $w$ is a weighted $k$-uniform hypergraph on vertex set $V$, where $|V|=n$.  
Let $\phi_i$ and $\phi_i^*$ be defined as in \eqref{canonical} and \eqref{canonicalstar}.
Then
\begin{equation}\label{nice2}
W_i=\E_\pi|\langle w_\pi,\phi_i^*\rangle|.
\end{equation}
\end{lemma}

\begin{proof}
For $i=0$, we have
\begin{align*}
W_0
&=\binom nk^{-1}|\sum_{e\in V^{(k)}}w(e)|\\
&=\binom nk^{-1}|\langle w,{\mathbf 1}\rangle|\\
&=\E_\pi\binom nk^{-1}|\langle w_\pi,{\mathbf 1}\rangle|\\
&=\E_\pi|\langle w_\pi,\phi_0^*\rangle|,
\end{align*}
as $\phi_0=\mathbf 1$, $\phi_0^*=\phi_0/\binom nk$
and $\langle w_\pi,\mathbf 1\rangle=\langle w,\mathbf 1\rangle$ for every $\pi$.

We now proceed by induction on $i$.  For $i\ge1$, we have $W_i=\E_{x,y}'W_{i-1}^{xy}$.  
Choose a sequence $x_1,y_1,\ldots,x_k,y_k$ of distinct vertices, and let
$W=V\setminus\{x_1,y_1\}$.  We define $\phi_i$, $\phi_i^*$ as in \eqref{canonical} and \eqref{canonicalstar},
and let $\psi_i$, $\psi_i^*$ be the corresponding functions for $W$ and the sequence $x_2,y_2,\ldots,x_k,y_k$.
Thus $\psi_{i-1}:W^{(k-1)}\to\R$ is given by
\begin{equation}\label{psican}
\psi_{i-1}(A)=\phi_{i-1}(x_2,y_2,\ldots,x_i,y_i;A)
\end{equation}
and
\begin{equation}\label{psicanstar}
\psi_{i-1}^*=\psi_{i-1}/\binom{(n-2)-2(i-1)}{(k-1)-(i-1)}=\psi_{i-1}/\binom{n-2i}{k-i}.
\end{equation}
Note that for $e\in W^{(k-1)}$, it follows from \eqref{canonical}, \eqref{canonicalstar},
\eqref{psican} and \eqref{psicanstar} that we have
\begin{equation}\label{psph}
\psi_{i-1}^*(e)=\phi_i^*(e\cup\{x\})=-\phi_i^*(e\cup\{y\}).
\end{equation}

It follows by induction from \eqref{nice2} that, writing $(W_0^{x_1y_1},\ldots,W_{k-1}^{x_1y_1})$ for the $W$-vector of $w^{x_1y_1}$,
\begin{equation}\label{indstep}
W^{x_1y_1}_{i-1}=\E_{\pi^*\in S(W)}|\langle (w^{x_1y_1})_{\pi^*},\psi_{i-1}^*\rangle|.
\end{equation}
We identify $\pi^*\in S(W)$ with the corresponding $\pi\in S(V)$ that fixes $x_1,y_1$ and otherwise acts as $\pi^*$.  
Then $(w_\pi)^{x_1y_1}=(w^{x_1y_1})_{\pi^*}$,
and so, by \eqref{psph},
\begin{align}
\langle (w^{x_1y_1})_{\pi^*},\psi_{i-1}^*\rangle
&=\sum_{e\in W^{(k-1)}}(w_\pi)^{x_1y_1}(e)\psi_{i-1}^*(e) \notag\\
&=\sum_{e\in W^{(k-1)}}(w_\pi(e\cup\{x_1\})-w_\pi(e\cup\{y_1\}))\psi_{i-1}^*(e) \notag\\
&=\sum_{e\in W^{(k-1)}} (w_\pi(e\cup\{x_1\})\phi_i^*(e\cup\{x_1\})+w_\pi(e\cup\{y_1\})\phi_i^*(e\cup\{y_1\})) \notag\\
&=\sum_{f\in V^{(k)}}w_\pi(f)\phi_i^*(f) \notag\\
&=\langle w_\pi,\phi_i^*\rangle, \label{wpps}
\end{align}
since we have $\phi_i(f)=0$ unless $f$ is compatible with the sequence of vertices $(x_1,y_1,\ldots,x_i,y_i)$.  

For $i\ge1$, we have
$$
W_i
=\E_{x,y}'W^{xy}_{i-1}\\
=\E_\pi W^{\pi(x_1)\pi(y_1)}_{i-1},
$$
since $(\pi(x_1),\pi(y_1))$ is uniformly distributed over distinct vertices $x,y$.  
Now for an edge $e\in(V\setminus\{\pi(x),\pi(y)\})^{(k-1)}$,
\begin{align*}
w^{\pi(x_1)\pi(y_1)}(e)
&=w(e\cup\{\pi(x_1)\})-w(e\cup\{\pi(y_1)\})\\
&=w_\pi(\pi^{-1}(e)\cup\{x_1\})-w_\pi(\pi^{-1}(e)\cup\{y_1\})\\
&=(w_\pi)^{x_1y_1}(\pi^{-1}(e)).
\end{align*}  
It follows that $w^{\pi(x_1)\pi(y_1)}$ and $(w_\pi)^{x_1y_1}$ are isomorphic and so have the same
$W$-vector $((W_\pi^{x_1y_1})_0,\ldots,(W_\pi^{x_1y_1})_{k-1})$.  Using \eqref{indstep} and \eqref{wpps}, we get
$$(W_\pi^{x_1y_1})_{i-1}=\E_{\rho^*\in S(W)}|\langle ((w_\pi)^{x_1y_1})_{\rho^*},\psi_{i-1}^*\rangle|
=\E_{\rho^*\in S(W)}|\langle (w_\pi)_{\rho},\phi_{i-1}^*\rangle|.$$
Thus
\begin{align*}
W_i
&=\E_\pi W^{\pi(x_1)\pi(y_1)}_{i-1}\\
&=\E_\pi(W_\pi^{x_1y_1})_{i-1}\\
&=\E_{\pi,\rho^*}|\langle(w_\pi)_{\rho^*},\phi_i^*\rangle|\\
&=\E_{\pi,\rho^*}|\langle w_{\rho\pi},\phi_i^*\rangle|\\
&=\E_\pi|\langle w_\pi,\phi_i^*\rangle|,
\end{align*}
as $\rho\pi$ is uniformly distributed over $S(V)$.  This gives \eqref{nice2}.
\end{proof}

We remark that \eqref{nice2} is reminiscent of taking a Fourier transform.

In order to show that our theorems do not give trivial bounds, we need to know that the $l_1$ norm of a weighting is preserved 
up to a constant factor by its $W$-vector.  This is the substance of the next result.

\begin{lemma}\label{thmb}
For every $k\ge0$ there are constants $c,c'>0$ such that the following holds.
For every $n\ge2k$, and every weighted $k$-uniform hypergraph $w$
on $[n]$, 
\begin{equation}\label{thmbeq}
cn^{-k}||w||_1 \le \sum_{i=0}^k W_i\le c'n^{-k}||w||_1,
\end{equation}
where $(W_0,\ldots,W_k)$ is the $W$-vector of $w$.
\end{lemma}

Let us first note the following.

\begin{proposition}\label{local}
For every $k\ge1$ there is $c>0$ such the following holds.
For every $n>k$, and every weighted $k$-uniform hypergraph $w$ with vertex set $[n]$ such that $w([n])=0$,
$$\E_{A,x,y}|w(A\cup\{x\})-w(A\cup\{y\})|\ge c_kn^{-k}||w||_1,$$
where the expectation is taken over $(k-1)$-sets $A$ and distinct vertices $x,y\not\in A$ chosen uniformly at random.
\end{proposition}

\begin{proof}
Since we can rescale, it is enough to
show that for every weighting $w$ on $[n]^{(k)}$ with $||w||_1=\binom nk$ and $w([n])=0$ we have
\begin{equation}\label{Af4k}
\E_{A,x,y}|w(A\cup\{x\})-w(A\cup\{y\})|\ge 1/4k.
\end{equation}
Note that $\E_e[w(e)^+]=\E_e[w(e)^-]$ and $\E_e|w(e)|=\E_e [w(e)^++w(e)^-]=1$,
so $\E_e[w(e)^+]=\E_e[w(e)^-]=1/2$.

Suppose that $w$ is nonnegative on $p\binom nk$ edges and negative on $(1-p)\binom nk$ edges.  We may assume $p\le1/2$ or work with $-w$.  
Pick with replacement random edges $e$ and $f$ 
and let $e_1\cdots e_i$ be a random shortest path between them
(so $e_1=e$, $e_i=f$ and each step replaces one element of the edge).  
If $e,f$ are distinct, we set $A=e_1\cap e_2$, and let $x,y$ be the remaining vertices of $e_1,e_2$ respectively; otherwise,
we choose $A,x,y$ uniformly at random.  Then $(A,x,y)$ is uniformly distributed,
and 
\begin{align}
\E_{A,x,y}|w(A\cup\{x\})-w(A\cup\{y\})|
&\ge \E_{e,f}|w(e_2)-w(e_1)| \notag \\
&\ge\E_{e,f}|w(e)-w(f)|/k. \label{Afk}
\end{align}
With probability $1-p$ we have $w(f)\le0$.  
Conditioning on this event, $\E_{e,f}[|w(e)-w(f)| \mid w(f)\le 0]\ge \E_{e,f}[w(e)^+\mid w(f)\le0]=\E_e[w(e)^+]=1/2$.
We conclude that (without conditioning)
$\E[|w(e)-w(f)|]\ge (1-p)\cdot(1/2)\ge1/4$, and \eqref{Af4k} then follows from \eqref{Afk}.
\end{proof}

\begin{proof}[Proof of Lemma \ref{thmb}]
Let us write $w=w_0+w_1$ where $w_0$ is constant and $\sum_{e}w_1(e)=0$.  

Clearly $||w||_1\le||w_0||_1+||w_1||_1$; we also have $||\phi_i||_1=2^i\binom{n-2i}{k-i}$,
since there are exactly $2^i\binom{n-2i}{k-i}$ edges compatible with any sequence $x_1,y_1,\ldots,x_i,y_i$.  Thus
$||\phi_i^*||_1=2^i$.  So, by Lemma \ref{niceform}, we have
\begin{align*}
W_i
&=\E_\pi|\langle w_\pi,\phi_i^*\rangle|\\
&=\E_\pi|\sum_ew_\pi(e)\phi_i^*(e)|\\
&\le\sum_e\E_\pi|w_\pi(e)\phi_i^*(e)|\\
&=\sum_e|\phi_i^*(e)|\cdot\E_\pi|w_\pi(e)|\\
&=\sum_e|\phi_i^*(e)|\cdot||w||_1/\binom nk\\
&=2^i||w||_1/\binom nk.
\end{align*}
Summing over $i$ gives the upper bound in \eqref{thmbeq}.

For the lower bound, note first that, by linearity and Proposition \ref{local}, for $k\ge1$,
\begin{align}
\E_{x,y}'||w^{xy}||_1
&=\E_{x,y}'\sum_{A\in (V\setminus\{x,y\})^{k-1}}|w(A\cup\{x\})-w(A\cup\{y\})|\notag\\
&=\binom{n-2}{k-1}{\E_{x,y,A}}'|w(A\cup\{x\})-w(A\cup\{y\})|\notag\\
&\ge c_k\binom{n-2}{k-1}n^{-k}||w_1||_1\notag\\
&\ge c_k'||w_1||_1/n,\label{eq10}
\end{align}
where $c$ and $c_k$ are constants depending only on $k$.

For $k=0$, we have $W_0=||w||_1$, giving \eqref{thmbeq} as required.
We now argue by induction on $k$: we will use $c$, $c'$, etc, for constants that depend only on $k$.
For $k\ge1$ and distinct $x,y\in V$, we have by induction
$$W_0^{xy}+\dots+W_{k-1}^{xy}\ge c||w^{xy}||_1/n^{k-1}.$$
Since $w_0$ is a constant function, we also have
$W_0=||w_0||_1/\binom nk\ge c''||w_0||_1/n^k$. 
It follows by \eqref{eq10} that
\begin{align*}
W_0+\cdots+W_k
&= W_0+\E_{x,y}'(W_0^{xy}+\dots+W_{k-1}^{xy})\\
&\ge c'(||w_0||_1/n^k+\E_{x,y}'||w^{xy}||_1/n^{k-1})\\
&\ge c''(||w_0||_1/n^k+||w_1||_1/n^k)\\
&\ge c'''||w||_1/n^k.
\end{align*}
\end{proof}

\section{Bounding in terms of transpositions}\label{trans}

In order to prove Theorems \ref{thdisc} and \ref{thexp}, we will need bounds both on $\disc(w,u)$ and on $\E_\pi|\langle w_\pi,u\rangle|$.  These will be driven by two results bounding these quantities from below in terms of the effects of single transpositions.

Let us fix the ground set $V$ and pick distinct vertices $x,y\in V$.  Let $\tau$ be the transposition $(xy)$.  Let $w$ and $u$ be two weightings of $V^{(k)}$, and choose uniformly at random two permutations $\pi, \sigma$.  We define
$\gamma(w,u)$ by
\begin{equation}\label{gamma}
\gamma(w,u)=\E_{\pi,\sigma}|\langle w_\pi,u_{\sigma}\rangle-\langle w_{\tau\pi}u_{\sigma}\rangle|.
\end{equation}
Thus $\gamma(w,u)$ measures the typical effect on the inner product of exchanging $x$ and $y$ in one copy of $V$.  
Note that $\gamma(w,u)$ does not depend on our choice of $x$ and $y$, 
since the expectation is taken over random permutations of the ground set for both $w$ and $u$.

Our bounds will depend on the following two lemmas.

\begin{lemma}\label{thmf}
For every $k\ge1$ there is $c>0$ such that the following holds.
For every $n\ge k$ and every pair $w,u$ of functions from $[n]^{(k)}$ to $\R$, we have 
\begin{equation}\label{thmfbound} 
\disc^+(w,u)\disc^-(w,u)\ge c^2 \gamma(w,u)^2 n^2.
\end{equation}
\end{lemma}

\begin{lemma}\label{thme}
For every $k\ge1$ there is $c>0$ such that the following holds.
For every $n\ge k$ and every pair $w,u$ of functions from $[n]^{(k)}$ to $\R$
$$\E_\pi|\langle w_\pi,u\rangle| \ge c \gamma(w,u) \sqrt n.$$
\end{lemma}

We start by setting up a framework for the proofs of Lemma \ref{thmf} and Lemma \ref{thme}.

Let $w,u$ be two weightings of $[n]^{(k)}$.
Let $I$ be an index set, and suppose we have transpositions 
\begin{equation*}\label{transpi}
\tau^i=(x_iy_i), \qquad i\in I
\end{equation*}
such that the pairs $\{x_i,y_i\}_{i\in I}$ are disjoint.
For $J\subset I$, we define $\tau^J$ to be the product of the transpositions $\{\tau^j:j\in J\}$ 
(note that the $\tau^j$ commute, and $\tau^J=(\tau^J)^{-1}$).  
We will want to consider the difference $|\langle w_{\tau^J},u\rangle -\langle w,u\rangle|$ for various sets $J$.   For $i\in I$, we define
$$\delta(i) =\langle w,u\rangle -\langle w_{\tau^i},u\rangle.$$
For $J\subset I$, we define
$$\delta(J)=\sum_{i\in J}\delta(i)$$
and
\begin{equation}\label{Delta}
\Delta(J)=\sum_{i\in J}|\delta(i)|. 
\end{equation}
If we want to specify $w,u$ explicitly, we will write $\Delta_{w,u}$ instead of $\Delta$, and so on.  However, we
drop indices when they are not necessary.

For a set $e\in V^{(k)}$, let
$$\tr(e)=\{i\in I:|e\cap\{x_i,y_i\}|=1\}.$$
Note that $i\in\tr(e)$ if and only if $\tau^i(e)\ne e$,  
and $\tr(\tau^J(e))=\tr(e)$ for any $J$. 
We decompose $\langle w,u\rangle-\langle w_{\tau^J},u\rangle$ as follows.

\begin{proposition}\label{decomp}
Let $n\ge k\ge1$, let $w,u$ be weightings on $[n]^{(k)}$, 
and let $I$ be an index set for transpositions $\tau^i$ as in \eqref{transpi}.
For every $J\subset I$, we have
$$\langle w,u\rangle-\langle w_{\tau^J},u\rangle=\frac12\sum_{e\in [n]^{(k)}}(w(e)-w(\tau^Je))(u(e)-u(\tau^Je)).$$
In particular, if $\tau=(xy)$, 
\begin{equation}\label{wwt}
\langle w,u\rangle-\langle w_\tau,u\rangle=\langle w^{xy},u^{xy}\rangle.
\end{equation}
\end{proposition}

\begin{proof}
Note that, for any permutation $\pi$ and any $f:V^{(k)}\to\R$,
we have $\sum_{e\in [n]^{(k)}}f(e)=\sum_{e\in [n]^{(k)}}f(\pi e)$.  So
\begin{align*}
\langle w,u\rangle-\langle w_{\tau^J},u\rangle
&=\sum_e(w(e)u(e)-w_{\tau^J}(e)u(e))\\
&=\frac12\sum_e(w(e)u(e)-w_{\tau^J}(e)u(e))\\
&\qquad+\frac12\sum_e(w(\tau^Je)u(\tau^Je)-w_{\tau^J}(\tau^Je)u(\tau^Je))\\
&=\frac12\sum_e(w(e)u(e)-w(\tau^Je)u(e))\\
&\qquad+\frac12\sum_e(w(\tau^Je)u(\tau^Je)-w(e)u(\tau^Je))\\
&=\frac12\sum_e(w(e)-w(\tau^Je))(u(e)-u(\tau^Je)),
\end{align*}
where all sums are over $[n]^{(k)}$.

To prove \eqref{wwt}, note that $w(e)-w(\tau e)=0$ unless $|e\cap\{x,y\}|=1$.  So, writing $W=[n]\setminus \{x,y\}$, 
\begin{align*}
\frac12 &\sum_{e\in[n]^{(k)}}(w(e)-w(\tau e))(u(e)-u(\tau e))\\
&=\frac12\sum_{f\in W^{(k-1)},v\in\{x,y\}} 
   (w(f\cup\{v\})-w(f\cup\{\tau v\}))(u(f\cup\{v\})-u(f\cup\{\tau v\}))\\
&=\sum_{f\in W^{(k-1)}} 
   (w(f\cup\{x\})-w(f\cup\{y\}))(u(f\cup\{x\})-u(f\cup\{y\}))\\
&=\sum_{f\in W^{(k-1)}} w^{xy}(f)u^{xy}(f)\\
&=\langle w^{xy},u^{xy}\rangle.
\end{align*}
\end{proof}

We also consider the expected effect of a randomly chosen set of transpositions.

\begin{lemma}\label{coeff}
Let $n\ge k\ge1$, let $w,u$ be weightings on $[n]^{(k)}$, 
and let $I$ be an index set for transpositions $\tau^i$ as in \eqref{transpi}.
Let $I$ be fixed, and let $J$ be a random subset of $I$, where each $i\in I$ is taken independently with probability $p$.  Then
$\E(\langle w,u\rangle-\langle w_{\tau^J},u\rangle)$ can be written as a polynomial in $p$ of the form
\begin{equation}\label{polycoeff}
\delta(I)p+\sum_{i=2}^kA_ip^i,
\end{equation}
for some real numbers $A_2,\dots,A_k$.
\end{lemma}

\begin{proof}
For $e\in[n]^{(k)}$, let
$$\mu_J(e)=(w(e)-w(\tau^Je))(u(e)-u(\tau^Je)).$$
By Proposition \ref{decomp} we have
\begin{equation}\label{12sum}
\langle w,u\rangle-\langle w_{\tau^J},u\rangle
=\frac12\sum_e\mu_J(e).
\end{equation}
For a given edge $e$, the value of $\mu_J(e)$ depends only on $J\cap\tr(e)$, so
$$\mu_J(e)=\mu_{J\cap \tr(e)}(e).$$

It follows that
$$\E\mu_J(e)=\sum_{A\subset \tr(e)}p^{|A|}(1-p)^{|\tr(e)|-|A|}\mu_A(e),$$
and hence, by \eqref{12sum}, $\E(\langle w,u\rangle-\langle w_{\tau^J},u\rangle)$ is a polynomial in $p$ with degree at most $k$.
Since $\mu_\emptyset(e)=0$, the constant term is 0.

Now consider the behaviour of $\langle w,u\rangle-\langle w_{\tau^J},u\rangle$ as $p\to0$.  For each $i\in J$ we have $\P(J={i})=p+O(p^2)$.  As $\P(|J|>1)=O(p^2)$, it follows that 
$$\E(\langle w,u\rangle-\langle w_{\tau^J},u\rangle)=p\sum_i\delta(i)+O(p^2),$$
and so the coefficient of $p$ is $\sum_{i\in I}\delta(i)=\delta(I)$.
\end{proof}

We now prove the two lemmas stated at the beginning of the section.

\begin{proof}[Proof of Lemma \ref{thmf}]
Adding a constant to $w$ or $u$ does not affect $\disc^+(w,u)$, $\disc^-(w,u)$ or $\gamma(w,u)$,
so we may assume that $w([n])=u([n])=0$. 
Note first that $\E_\pi\langle w_\pi,u\rangle=0$ and so
$\E_\pi[\langle w_\pi,u\rangle^+]=\E_\pi[\langle w_\pi,u\rangle^-]=\E_\pi|\langle w_\pi,u\rangle|/2$.  

For fixed $n$, we can argue as follows. From \eqref{gamma} we have 
\begin{align*}
\gamma(w,u)
&=\E_{\pi,\sigma}|\langle w_\pi,u_\sigma\rangle-\langle w_{\tau\pi},u_\sigma\rangle|\\
&\le \E_{\pi,\sigma}(|\langle w_\pi,u_\sigma\rangle|+|\langle w_{\tau\pi},u_\sigma\rangle|)\\
&=2\E_\pi|\langle w_\pi,u\rangle|,
\end{align*}
and so we have
\begin{align*}
\min \{\disc^+(w,u),\disc^-(w,u)\}
&\ge \min\{\E_\pi[\langle w_\pi,u\rangle^+],\E_\pi[\langle w_\pi,u\rangle^-]\}\\ 
&\ge \gamma(w,u)/4 =\frac12\E_\pi|\langle w_\pi,u\rangle|.
\end{align*}
It follows that \eqref{thmfbound} holds for any fixed constant $n$ (and appropriate $c>0$), and so 
we may assume that $n>100k$.

Let $K\ge2$ be a fixed constant (which we will specify later), and
suppose that $\disc^+(w,u)\le\disc^-(w,u)$.  
If $\disc^+(w,u)\ge\gamma(w,u)n/10K$, we are done (with $c=1/10K$), so we may assume that
\begin{equation}\label{upper1}
\disc^+(w,u)=\frac{\gamma(w,u)n}{10\alpha}
\end{equation}
for some $\alpha\ge K$.  We shall show that, for some (small) constant $c>0$, we have
$$\disc^-(w,u)\ge c\gamma(w,u)\alpha n.$$

Let $t=\lfloor n/2\rfloor$ and let $x_1,\ldots,x_t,y_1,\ldots,y_t$ be a sequence of distinct vertices of $V$.  
For $I=\{1,\ldots,t\}$ and each $i\in I$, let $\tau^i=(x_iy_i)$. 
Let $\pi$ and $\sigma$ be chosen independently and uniformly at random from $S_n$.
Then \eqref{gamma} and linearity of expectation imply that 
$$\E\Delta_{w_\pi,u_\sigma}(I)=t\gamma(w,u).$$
Let $I^+=\{i:\delta_{w_\pi,u_\sigma}(i)>0\}$ and $I^-=\{i:\delta_{w_\pi,u_\sigma}(i)\le0\}$,
so $\delta_{w_\pi,u_\sigma}(I)=\delta_{w_\pi,u_\sigma}(I^+)+\delta_{w_\pi,u_\sigma}(I^-)
=\Delta_{w_\pi,u_\sigma}(I^+)-\Delta_{w_\pi,u_\sigma}(I^-)$.  Since $\E\delta_{w_\pi,u_\sigma}(I)=0$ we have
$$\E\Delta_{w_\pi,u_\sigma}(I^+)=\E\Delta_{w_\pi,u_\sigma}(I^-)=\frac t2\gamma(w,u).$$
We also have $\E\langle w_\pi,u_\sigma\rangle=0$, and so 
$$\E_{\pi,\sigma}[\alpha\langle w_\pi,u_\sigma\rangle+\Delta(I^+)]=\frac{t}{2}\gamma(w,u).$$
We can therefore choose $\pi$, $\sigma$ such that $\alpha\langle w_\pi,u_\sigma\rangle+\Delta_{w_\pi,u_\sigma}(I^+)\ge t\gamma(w,u)/2$.
Replacing $w$ and $u$ by $w_\pi, u_\sigma$, we may therefore assume that
\begin{equation}\label{choice}
\alpha\langle w,u\rangle+\Delta_{w,u}(I^+)\ge\frac t2\gamma(w,u).
\end{equation}
Note that this replacement does not change the value of $\disc^+$ or of $\disc^-$.

Now consider the effects of applying $\tau^J$, where $J\subset I^+$ is a random subset of $I^+$ 
with each $i\in I^+$ is present independently with probability $p$.  
Lemma \ref{coeff} tells us that
\begin{equation}\label{wdiff}
\E[\langle w_{\tau^J},u\rangle-\langle w,u\rangle]=p\delta(I^+)+\sum_{i=2}^kA_ip^i
\end{equation}
for some $A_2,\ldots,A_k$.  It follows from \eqref{choice},
by considering the case $p=1/\alpha$, that we have
\begin{align}
\E\langle w_{\tau^J},u\rangle
&=\langle w,u\rangle + \E[\langle w_{\tau^J},u\rangle-\langle w,u\rangle]\notag\\
&=\langle w,u\rangle + \Delta(I^+)/\alpha + \sum_{i=2}^kA_i/\alpha^i\notag\\
&\ge\frac{t}{2\alpha}\gamma(w,u)+\sum_{i=2}^k{A_i}/{\alpha^i}. \label{wtju1}
\end{align}
Now \eqref{upper1} implies that $\langle w_{\tau^J},u\rangle\le \gamma(w,u)n/10\alpha$ for any choice of $J$, so \eqref{wtju1} implies that
\begin{align*}
\sum_{i=2}^k A_i/\alpha^i
&\le \gamma(w,u)n/10\alpha-t\gamma(w,u)/2\alpha\\
&<-\gamma(w,u)n/10\alpha.
\end{align*}
Since $\alpha\ge2$, we must have $|A_i|\ge \alpha\gamma(w,u)n/20$ for some $i\ge2$. 

It follows by \eqref{poly1} that
for some $p\in[0,1]$ and some $c_k>0$ that depends only on $k$,
we have
\begin{equation}\label{escape}
|p\delta(I^+)+\sum_{i=2}^kA_ip^i|\ge 2c_k\alpha\gamma(w,u)n.
\end{equation}
But now, choosing $p$ such that \eqref{escape} holds, 
we have by \eqref{wdiff}
\begin{equation}\label{diff1}
\E|\langle w,u\rangle-\langle w_{\tau^J},u\rangle|\ge 2c_k\alpha\gamma(w,u)n
\end{equation}
and, since \eqref{upper1} holds, if we have chosen $K>1/\sqrt{c_k}$, we must have,
by \eqref{upper1} and \eqref{diff1},
\begin{align*}
\E\min(\langle w,u\rangle,\langle w_{\tau^J},u\rangle)
&=\E[\max(\langle w,u\rangle,\langle w_{\tau^J},u\rangle)-|\langle w,u\rangle-\langle w_{\tau^J},u\rangle|]\\
&\le \frac{\gamma(w,u)n}{10\alpha}-2c_k\alpha\gamma(w,u)n\\
&= \gamma(w,u)n\alpha\cdot[1/10\alpha^2-2c_k]\\
&\le -c_k\alpha\gamma(w,u)n.
\end{align*}
In particular, there is some $J$ such that
$$\langle w_{\tau^J},u\rangle \le -c_k\alpha\gamma(w,u)n$$
and so $\disc^-(w,u)\ge c_k\alpha\gamma(w,u)n$, as claimed.
\end{proof}

\begin{proof}[Proof of Lemma \ref{thme}]
We would like to argue as in the proof of Lemma \ref{thmf}.  However, there is an important difference: 
in the previous proof we could replace $w$ and $u$ by our choice of $w_\pi$ and $u_\sigma$, 
and then choose an advantageous set of transpositions to apply; 
now we must select our permutations $\pi, \sigma$ and transpositions so that the resulting permutations are uniformly distributed.

Consider first a specific choice of $\pi$ and $\sigma$, 
and let $t$, $I$ and the transpositions $\tau^i$ be defined as in the proof of Lemma \ref{thmf}.  
Recall that the set $I$ and the transpositions $\tau^i$ are fixed with respect to the ground set. 
However, $\delta(I)$ might be close to 0, which is not helpful if we want to use \eqref{polycoeff}.  
We therefore generate random sets $J_0$ and $J_1$ of transpositions in two steps as follows:
\begin{enumerate}
\item Let $J_0$ be a random subset of $I$, chosen uniformly at random from all $2^{|I|}$ subsets.
\item Let $p\in[0,1]$ be chosen uniformly at random, and let $J_1\subset J_0$ be a random subset, 
where each $i\in J_0$ is taken independently with probability $p$.
\end{enumerate}

Consider first $J_0$.  It follows from Lemma \ref{sumn} that
\begin{equation}\label{part1}
\E_{J_0}|\delta(J_0)|\ge\frac{\Delta(I)}{\sqrt{8|I|}}.
\end{equation}
Now if we condition on $J_0$ and $p$, then by Lemma \ref{coeff} we have
$$\E_{J_1}[\langle w,u\rangle-\langle w_{\tau^{J_1}},u\rangle \mid J_0,p]=\delta(J_0)p+\sum_{i=2}^kA_ip^i,$$
for some $A_2,\ldots,A_k$ that depend on $J_0$.  It then follows from Lemma \ref{poly} 
and the tower law for expectation that
there is a constant $c_k>0$ such that, if we condition just on $J_0$, we have
$$\E_{p,J_1}[|\langle w,u\rangle-\langle w_{\tau^{J_1}},u\rangle| \mid J_0]
=\E_p[|\delta(J_0)p+\sum_{i=2}^kA_ip^i|]
\ge c_k|\delta(J_0)|.$$
But now by \eqref{part1} and (again) the tower law for expectation it follows that
$$\E_{J_0,p,J_1}[|\langle w,u\rangle-\langle w_{\tau^{J_1}},u\rangle|]\ge\frac{c_k\Delta(I)}{\sqrt{8|I|}}.$$

This bound holds for any fixed placement of $w$ and $u$.  However, with a uniformly random choice of permutations $\pi$ and $\sigma$, 
giving weightings $w_\pi$ and $w_\sigma$, we have (by definition from \eqref{gamma} and \eqref{Delta})
$$\E_{\pi,\sigma}\Delta_{w_\pi,u_\sigma}(I)=t\gamma(w,u).$$
Thus there is a constant $c=c(k)>0$ such that
$$\E_{\pi,\sigma,J_0,p,J_1}[|\langle w_\pi,u_\sigma\rangle-\langle w_{\tau^{J_1}\pi},u_\sigma\rangle|]
\ge \frac{c_kt\gamma(w,u)}{\sqrt{8t}}
\ge 2c\gamma(w,u)\sqrt{n}.$$
By the triangle inequality, we have
\begin{align*}
2c\gamma(w,u)\sqrt{n}
&\le \E_{\pi,\sigma,J_0,p,J_1}[|\langle w_\pi,u_\sigma\rangle|+|\langle w_{\pi\tau^J},u_\sigma\rangle|]\\
&=2\E_{\pi,\sigma}|\langle w_\pi,u_\sigma\rangle|\\
&=2\E_\pi|\langle w_\pi,u\rangle|,
\end{align*}
which implies our result.
\end{proof}

\section{Proof of Theorems \ref{thdisc} and \ref{thexp}}\label{proofs1}

We are now ready to prove our main quantitative results.  We begin by proving Theorem \ref{thexp};
Theorem \ref{thdisc} will then follow easily.  At the end of the section, we will deduce another result
on partitioning families of hypergraphs with large pairwise discrepancy.

\begin{proof}[Proof of Theorem \ref{thexp}]
As usual, we write $w=w_0+w_1$ and $u=u_0+u_1$, where $u_0$, $w_0$ are constant functions and 
$u_1$, $w_1$ sum to 0.  
Since $\langle (w_0)_\pi,u_0\rangle=\langle w_0,u_0\rangle=\langle w_0,u\rangle$
and $\langle(w_1)_\pi,u_0\rangle=\langle w_0,(u_1)_\pi\rangle=0$ for any $\pi$, it follows from
Lemma \ref{thme} and Proposition \ref{simpleprop} that, for some $c'=c'(k)>0$,
\begin{align*}
\E_\pi|\langle w_\pi,u\rangle|
&=\E_\pi|\langle w_0,u_0\rangle+\langle (w_1)_\pi,u_1\rangle|\\
&\ge\max\{|\langle w_0,u_0\rangle|,\frac12\E_\pi|\langle (w_1)_\pi,u_1\rangle|\}\\
&\ge2c'\max\{n^kU_0W_0,\gamma(w_1,u_1)\sqrt{n}\}\\
&\ge c'n^kU_0W_0+c'\gamma(w_1,u_1)\sqrt{n}.
\end{align*}
It is therefore enough to prove that, for some fixed $c=c(k)>0$,
\begin{equation}\label{gammabound}
\gamma(w_1,u_1)\ge cn^{k-1/2}\sum_{i=1}^k n^{-i/2}W_iU_i.
\end{equation}
Note that $\gamma(w_1,u_1)=\gamma(w,u)$, since $u_0$ and $w_0$ are invariant under permutations.

We know from \eqref{wwt} that, with $\tau=(xy)$,
\begin{align*}
\gamma(w,u)
&=\E_{\pi,\sigma}|\langle w_\pi,u_\sigma\rangle-\langle w_{\tau\pi},u_\sigma\rangle|\\
&=\E_{\pi,\sigma}|\langle (w_\pi)^{xy},(u_\sigma)^{xy}\rangle|.
\end{align*}

For $k=1$, $(w_\pi)^{xy}$ and $(u_\sigma)^{xy}$ are nullary functions with absolute values $|w_\pi(x)-w_\pi(y)|$ and $|u_\sigma(x)-u_\sigma(y)|$.  
It follows from \eqref{w1def} that
\begin{align*}
\gamma(w,u)
&=\E_{\pi,\sigma}|(w_\pi(x)-w_\pi(y))(u_\sigma(x)-u_\sigma(y))|\\
&=\E_{\pi}|w_\pi(x)-w_\pi(y)|\cdot\E_{\sigma}|u_\sigma(x)-u_\sigma(y)|\\
&=\E_{a,b}'|w(a)-w(b)|\cdot\E_{a,b}'|u(a)-u(b)|\\
&= W_1U_1,
\end{align*}
as required.

For $k\ge 2$, we prove Theorem \ref{thexp} by induction.  Consider
$$\E_{\pi,\sigma}|\langle (w_\pi)^{xy},(u_\sigma)^{xy}\rangle|.$$
For fixed $x,y$ we shall (as usual) write $\E_{\rho^*}$ for the expectation over permutations $\rho^*$ of $[n]\setminus{x,y}$; we will identify each
such permutation $\rho^*$ with the corresponding permutation $\rho$ of $[n]$ that fixes $x$ and $y$ and otherwise acts as $\rho^*$.
Note that if $\pi\in S(V)$ and $\rho^*\in S(V\setminus\{x,y\})$ are both uniformly distributed, then so is $\rho\pi$.  So
$$
\E_{\pi,\sigma}|\langle (w_\pi)^{xy},(u_\sigma)^{xy}\rangle|
=\E_{\pi,\rho^*,\sigma}|\langle (w_{\rho\pi})^{xy},(u_\sigma)^{xy}\rangle|
$$

We know by induction that
$$\E_{\rho^*}|\langle (w^{xy})_{\rho^*},u^{xy}\rangle|\ge c_{k-1}\sum_{i=0}^{k-1}W^{xy}_iU^{xy}_in^{k-1-i/2},$$
where $(W^{xy}_0,\ldots,W^{xy}_{k-1})$ 
and $(U^{xy}_0,\ldots,U^{xy}_{k-1})$ 
are the $W$-vector of $w^{xy}$ and $u^{xy}$ respectively.
Now note that if $\rho$ fixes $x$ and $y$ then $(w^{xy})_{\rho^*}=(w_\rho)^{xy}$.
It follows that
\begin{align}
\E_\pi|\langle w_\pi,u\rangle-\langle w_{\tau\pi},u\rangle|
&=\E_\rho\E_{\rho^*}|\langle w_{\rho\pi},u\rangle-\langle w_{\tau\rho\pi},u\rangle|\notag\\
&=\E_\pi\E_{\rho^*}|\langle (w_{\rho\pi})^{xy},u^{xy}\rangle|\notag\\
&=\E_\pi\E_{\rho^*}|\langle (w_{\pi})^{xy}_{\rho^*},u^{xy}\rangle|\notag\\
&\ge c_{k-1}\E_\pi \sum_{i=0}^{k-1}(W_\pi^{xy})_iU^{xy}_in^{k-1-i/2},\label{epi}
\end{align}
where we have used \eqref{wwt}, and the fact that $\pi$, $\sigma\pi$, $\sigma^*\pi$ are all uniformly distributed over $S(V)$.
But $w_\pi^{xy}$ is isomorphic to $w^{\pi^{-1}(x)\pi^{-1}(y)}$, and so $(W_\pi^{xy})_i=(W^{\pi^{-1}(x)\pi^{-1}(y)})_i$.  
It follows that
\begin{equation}\label{epi2}
\E_\pi(W_\pi^{xy})_i=\E_{v,w}'(W^{vw})_i=W_{i+1}
\end{equation}
and so, by \eqref{epi},
\begin{align*}
\E_\pi|\langle w_\pi,u\rangle-\langle w_{\tau\pi},u\rangle|
&\ge c_{k-1}\sum_{i=0}^{k-1}W_{i+1}U^{xy}_in^{k-1-i/2}\\
&=c_{k-1}\sum_{i=1}^kW_iU^{xy}_{i-1}n^{k-1/2-i/2}.
\end{align*}
But now, applying the same argument to $u_\sigma$ over random $\sigma$ gives
\begin{align*}
\E_{\pi,\sigma}|\langle w_\pi,u_\sigma\rangle - \langle w_{\tau\pi},u_\sigma\rangle|
&\ge\E_\sigma c_{k-1}\sum_{i=1}^kW_i(U_\sigma^{xy})_{i-1}n^{k-1/2-i/2}\\
&\ge c_{k-1}'\sum_{i=1}^{k}W_iU_in^{k-1/2-i/2},
\end{align*}
where we have used \eqref{epi2} with $u$ instead of $w$ in the final line.
This proves inequality \eqref{gammabound}, and therefore \eqref{expineq}.
\end{proof}

\begin{proof}[Proof of Theorem \ref{thdisc}]
This now follows easily from Lemma \ref{thmf} and \eqref{gammabound}, 
as we have 
\begin{align*}
\disc^+(w,u)\disc^-(w,u)
&\ge c^2\gamma(w,u)^2n^2\\
&\ge c'n^{2k+1}(\sum_{i=1}^kn^{-i/2}W_iU_i)^2\\
&\ge c'n^{2k+1}\sum_{i=1}^kn^{-i}W_i^2U_i^2,
\end{align*}
since all terms in the sum are nonnegative.
\end{proof}

Theorems \ref{thdisc} and \ref{thexp} also allow us to prove the following result.

\begin{theorem}\label{thma1}
For every $k\ge1$ there are constants $c_1,c_2>0$ such that the following holds.
For every $n\ge 2k$, and every family
$\mathcal F$ of $k$-uniform hypergraphs with vertex set $[n]$
such that $w([n])=0$ for all $w\in\mathcal F$, there is a partition
$\mathcal F=\mathcal F_1\cup\cdots\cup\mathcal F_k$ such that, for every $1\le i\le k$ and all
distinct pairs $w,u\in \mathcal F_i$, we have
$$\disc(w,u) \ge c_1 ||w||_1||u||_1/n^{i/2-1/2+k},$$
and, 
$$\E_\pi|\langle w_\pi,u\rangle|\ge c_2||w||_1||u||_1/n^{i/2+k}.$$
\end{theorem}

\begin{proof}
Suppose $w\in \mathcal F$ has $W$-vector $(W_0,\ldots,W_k)$.  Since $w(V)=0$ we have $W_0=0$,
and so by Lemma \ref{thmb} there is $i\ge1$ with $W_i\ge||w||_1/n^k$: we choose such an $i$ and
place $w$ in $\mathcal F_i$.  Now for $w,u\in\mathcal F_i$, with $W$-vectors
$(W_0,\ldots,W_k)$ and $(U_0,\ldots,U_k)$, we have by Theorem \ref{thdisc}
$$\disc^+(w,u)\disc^-(w,u)\ge cn^{2k+1-i}W_i^2U_i^2\ge c'n^{1-i}||w||_1^2||u||_1^2/n^{2k}.$$
The first bound now follows, and the second follows similarly by applying Theorem \ref{thexp}.
\end{proof}

\section{Orthogonal sets of weightings}\label{proofs2}

Consider integers $n,k$ with $n\ge 2k\ge1$ and a set $V=\{v_1,\ldots,v_n\}$.  
Let us choose a sequence $\mathbf s=(x_1,y_1,\ldots,x_k,y_k)$ of elements of $V$ and
define the weightings $\phi_i$ on $V^{(k)}$ 
as in \eqref{canonical}.

We also define the subspace $V_i$ of $\R^{V^{(k)}}$ to be the linear span
$$V_i=\langle(\phi_i)_\pi:\pi\in S(V)\rangle.$$
Note that $V_i$ is independent of our choice of $\mathbf s$.

\begin{theorem}\label{orthogonal}
\begin{itemize}
 \item[(a)] For $i\ge1$, $\disc(w,\phi_i)=0$ if and only if $w\in V_i^\perp$. 
 \item[(b)] If $u\in V_i$ and $w\in V_j$, where $i\ne j$, then $\disc(u,w)=0$.
 \item[(c)] $\R^{V^{(k)}}$ is the direct sum $V_0\oplus\cdots\oplus V_k$.
 \item[(d)] Suppose that $u=u_0+\cdots+u_k$, with $u_i\in V_i$ for each $i$,
 and let $(U_0,\ldots,U_k)$ be the $W$-vector of $u$. 
 For $i=0,\ldots,k$, we have $U_i=0$ if and only if $u_i=\mathbf 0$.
 \item[(e)] If $w_1,\ldots,w_t$ are nonzero and satisfy $\disc(w_i,w_j)=0$ for all $i\ne j$ then there is a partition $[k]=I_1\cup\cdots\cup I_t$ such that we have
$$w_i\in V_0\oplus\bigoplus_{h\in I_i}V_h$$ 
for each $i$.  
\end{itemize}
\end{theorem}

\begin{proof}
(a) Note first that $\phi_i(V)=0$, so $\phi$ has density $d(\phi_i)=0$.  But then
\begin{align*}
\disc\langle u,\phi_i\rangle=0
&\iff \forall\pi, \quad\langle u_\pi,\phi_i\rangle=0\\
&\iff \forall\pi, \quad\langle u,(\phi_i)_\pi\rangle=0\\
&\iff \forall \mbox{ sequences $(\lambda_\pi)$, } \langle u,\sum_\pi\lambda_\pi(\phi_i)_\pi\rangle=0\\
&\iff u\in V_i^\perp.
\end{align*}

(b) We may assume that $i,j\ne0$ or else the result is trivial.  We may therefore assume $d(w)=d(u)=0$.
It is then sufficient to show that, for any choice of permutations $\pi$ and $\rho$, we have $\langle (\phi_i)_\pi,(\phi_j)_\rho\rangle=0$.
So, let us choose $\pi$ and $\rho$, and set $\psi_i=(\phi_i)_\pi$ and $\psi_j=(\phi_j)_\rho$.

For $1\le s\le i$, we let $(a_s,b_s)=(\pi^{-1}(x_s),\pi^{-1}(y_s))$, 
and, for $1\le t\le j$, we let $(c_t,d_t)=(\rho^{-1}(x_t),\rho^{-1}(y_t))$.  Then 
we have 
\begin{equation}\label{psii}
\psi_i(e)=
\begin{cases}
0&\mbox{if $|e\cap\{a_s,b_s\}|\ne1$ for some $1\le s\le i$}\\
(-1)^{|A\cap \{b_1,\ldots,b_i\}|}&\mbox{otherwise,}
\end{cases}
\end{equation}
and
\begin{equation}\label{psij}
\psi_j(e)=
\begin{cases}
0&\mbox{if $|e\cap\{c_t,d_t\}|\ne1$ for some $1\le t\le j$}\\
(-1)^{|A\cap \{d_1,\ldots,d_j\}|}&\mbox{otherwise.}
\end{cases}
\end{equation}

Now consider the multigraph $G$ with vertex set $V$ and edge set given by  $a_1b_1,\ldots,a_ib_i,c_1d_1,\ldots,c_jd_j$.  
As $E(G)$ is the union of two matchings, it contains no odd cycles and so is the vertex-disjoint union of paths and even cycles 
(possibly including double edges).  Even cycles and paths with an even number of edges meet $\{a_1b_1,\ldots,a_ib_i\}$ 
and $\{c_1d_1,\ldots,c_jd_j\}$ in the same number of edges, so (as $i\ne j$) there must be a path $P=x_1\cdots x_{2t}$ with an odd number of edges.
Let $X=V(P)$.  If $e\in V^{(k)}$ is such that $\psi_i(e)$ and $\psi_j(e)$ are both nonzero, it follows from \eqref{psii} and \eqref{psij} 
that either $e\cap X=\{x_1,x_3,\ldots,x_{2t-1}\}$ or $e\cap X=\{x_2,x_4,\ldots,x_{2t}\}$, as each edge of $P$ must contain exactly one vertex of $e$. 
Furthermore, as $P$ has $2t-1$ edges, if we write $e'=e\bigtriangleup X$ then
$$\psi_i(e')\psi_j(e')=(-1)^{2t-1}\psi_i(e)\psi_j(e)=-\psi_i(e)\psi_j(e).$$
It follows that $\psi_i(e)\psi_j(e)+\psi_i(e')\psi_j(e')=0$.  But now, pairing off such edges, we see that $\sum_e\psi_i(e)\psi_j(e)=0$ and so
\begin{equation}\label{orthog}
\langle\psi_i,\psi_j\rangle=0
\end{equation}
as required.  Note also that \eqref{orthog} holds if $i=0$ or $j=0$.

(c) It follows from \eqref{orthog} and linearity that there is no linear dependence among sets of vectors chosen from distinct $V_i$ and thus that $V_0+\cdots+V_k$ is a direct sum.  Now suppose that $u\in(V_0\oplus \cdots \oplus V_k)^\perp$.  For $i=0,\ldots,k$, and any $\pi$, we have
$\langle u,(\phi_i)_\pi\rangle=0$.  It follows that $\disc(u,\phi_i)=0$ for every $i$,
and so, by Lemma \ref{niceform}, $u$ has W-vector $(0,\ldots,0)$. 
But by Lemma \ref{thmb}, this implies that $u=\mathbf 0$.  It follows that $V_0\oplus\cdots\oplus V_k=\R^{V^{(k)}}$.

(d) For $j\ne i$, we have $V_j\subseteq V_i^\perp$
and so $\langle (u_i)_\pi,\phi_j\rangle=0$ for every $\pi$.
Thus, for any $j$,
\begin{equation}\label{jj}
U_j=\E_\pi|\langle (\sum_i u_i)_\pi,\phi_j\rangle|=\E_\pi|\langle (u_j)_\pi,\phi_j\rangle|.
\end{equation}
Now let $(U^j_0,\ldots,U^J_k)$ be the $W$-vector of the weighted hypergraph $u_j$, so \eqref{jj} implies $U_j=U^j_j$. 
Clearly $U^j_i=0$ if $i\ne j$. By Lemma \ref{thmb} we have $U^j_j\ne0$
if and only if $u_j\ne\mathbf0$.  

(e)
Suppose we have nonzero $k$-uniform weighted hypergraphs $u,w$ such that $\disc(u,w)=0$.  By (c), we can write $u=u_0+\cdots+u_k$ and $w=w_0+\cdots+w_k$, where $u_i,w_i\in V_i$ for each $i$.  If $u_i\ne0$ then
$\E_\pi|\langle u_\pi,\phi_i\rangle|=\E_\pi|\sum_j\langle(u_j)_\pi,\phi_i\rangle|=\E_\pi|\langle(u_i)_\pi,\phi_i\rangle|>0$,
since $\langle(u_j)_\pi,\phi_i\rangle=0$ for $j\ne i$.
It follows that $U_i>0$ whenever $u_i$ is nonzero, and similarly $W_i>0$ whenever $w_i$ is nonzero. 
Since $\disc(u,w)=0$, it follows from Theorem \ref{thdisc} that $U_iW_i=0$ for every $i\ge1$ 
and so we deduce that $u_i$ and $w_i$ cannot both be nonzero.  The result follows.
\end{proof}

Note in particular, that part (e) proves Theorem \ref{orthogset}.  Indeed it has the following stronger corollary.

\begin{corollary}
Suppose that $u_1,\ldots,u_k$ are weighted $k$-uniform hypergraphs on vertex set $V$
such that $u_i(V)=0$ for all $i$, and
$\disc(u_i,u_j)=0$ for all $0\le i<j\le k$.  Then there is a relabelling such that
$w_i\in V_i$ for each $i$.
\end{corollary}

\section{Further questions}\label{conc}

In this paper, we have proved some results on the discrepancy of pairs of weighted $k$-uniform hypergraphs.  However, many interesting questions remain.  

\begin{itemize}
\item What can we say about discrepancy of directed graphs, or more generally of directed $k$-uniform hypergraphs
(in which edges are ordered $k$-tuples of distinct vertices)? 
More simply, what about oriented graphs, or tournaments?
What can be said about functions from $X\times Y$ to $\R$, where we are allowed to permute both $X$ and $Y$?
\item It is interesting to note what Theorem \ref{thdisc} says about random hypergraphs.  Let us fix $p\in (0,1)$ and let $w$ be a random $k$-uniform hypergraph with vertex set $[n]$, where each edge is present independently with probability $p$.  
For $i\le k$, consider a sequence $x_1,y_1,\dots,x_i,y_i$ of $2i$ distinct vertices.  
It follows from \eqref{canonical} that $\phi_i$ is nonzero on 
$\Theta(n^{k-i})$  $k$-sets, while $\phi_i([n])=0$.  Thus $\langle \phi_i,w\rangle$ is the difference of two binomial random variables with (the same) distribution with parameters $\Theta(n^{k-i})$ and $p$.  It follows that
$\E|\langle \phi_i,w\rangle|=\Theta(\sqrt{p(1-p)}n^{(k-i)/2})$ and so 
$\E W_i=\Theta(\sqrt{p(1-p)}n^{(k-i)/2}/n^{k-i})=\Theta(\sqrt{p(1-p)}n^{-(k-i)/2})$.
If $w$ and $u$ are random graphs with densities $p$, $q$ respectively, it then follows from
Theorem \ref{thdisc} that 
$\E \disc(w,u) \ge \sqrt{p(1-p)q(1-q)} n^{(k+1)/2}$, where the main contribution comes from the final 
components $W_k$, $U_k$ of the $W$-vectors of $w$, $u$.  
The problem of determining the behaviour of $\E\disc(w,u)$ was raised for graphs in \cite{BS11};
stronger results can be found in
 Bollob\'as and Scott \cite{BS} and
Ma, Naves and Sudakov \cite{JNS}.
\item What about the sharpness of everything?  For instance, when are the bounds in Theorem \ref{thma0} and Theorem \ref{thdisc} sharp to within a constant factor? 
\item The results in this paper are concerned with weighted $k$-uniform hypergraphs, 
but what happens if we restrict ourselves to the unweighted case?  
We can always generate a pair of $k$-uniform hypergraphs with discrepancy 0 by letting $G$ be the hypergraph with all edges containing
a fixed vertex, and letting $H$ be any regular $k$-uniform hypergraph, but what if we want $G$ and $H$ to have density bounded away from 0 and 1? 
For $k=2$,  the lower bound \eqref{prod2} (from \cite{BS11}) shows that the discrepancy must be large;
but for $k=3$, as noted in the introduction,
there is a pair of dense unweighted hypergraphs with discrepancy 0.  What happens for $k\ge4$?  
Could a version of Conjecture 10 from \cite{BS11} hold in this case?
In the opposite direction, it would be very interesting to characterize zero discrepancy pairs of unweighted hypergraphs.  In light of Theorem \ref{thdisc}, one line of attack would be to consider which components of the $W$-vector can be 0 for an unweighted hypergaph.  More generally, which subsets of components can support the $W$-vector of an unweighted $k$-uniform hypergraph? 
And is there a set of three unweighted hypergraphs that pairwise have discrepancy 0?
\item Can we say anything about the {\em distribution} of $\langle w_\pi,u\rangle$ for $k\ge3$?  It seems natural to hope for some form of Central Limit Theorem, as in the cases $k=1,2$.  Perhaps less ambitiously: we have a lower bound on $\E|\langle w_\pi,u\rangle|$, but what about an upper bound?  Maybe it is possible to determine this expectation up to a $\Theta(1)$ factor.
\item To what extent do the results above extend to the continuous setting, when we have measurable functions from $[0,1]^k$ to $\R$?  
\item We have worked with real weights in this paper.  What happens if we work with complex functions?
\item It would be interesting to consider different group actions.  As a starting point, what happens if we take the action of the cyclic group on itself, or of ${\mathbb Z}_2^n$ on itself?  In the case of the cyclic group, it would be natural to work with complex weights, as the Fourier basis gives a set of pairwise orthogonal weightings.  
\end{itemize}

\end{document}